\documentclass[10pt,a4paper]{article}

\usepackage[hyperfootnotes=false,pdfpage layout=SinglePage]{hyperref}
\usepackage{mathptmx}
\usepackage[scaled=.92]{helvet}
\usepackage{courier}
\usepackage{amssymb}
\usepackage{amsmath}
\usepackage{amsthm}
\usepackage{latexsym}
\usepackage{amsfonts}
\usepackage{mathrsfs}
\usepackage{setspace}

\newtheorem{thm}{Theorem}[section]
\newtheorem{cor}[thm]{Corollary}

\newtheorem{prop}[thm]{Proposition}
\theoremstyle{definition}

\newtheorem{ex}[thm]{Example}

\numberwithin{equation}{section}
\newcommand{\paragraf}{\textsection}

\newcommand{\R}{\ensuremath{\mathbb R}}    
\newcommand{\C}{\ensuremath{\mathbb C}}    
\newcommand{\N}{\ensuremath{\mathbb N}}    

\newcommand{\gperp}{{[\perp]}}
\newcommand{\iso}{\circ}
\newcommand{\product}{[\cdot\,,\cdot]}
\newcommand{\hproduct}{(\cdot\,,\cdot)}

\newcommand{\calD}{\mathcal D}

\newcommand{\calH}{\mathcal H}

\newcommand{\calK}{\mathcal K}         
\newcommand{\calL}{\mathcal L}         
\newcommand{\calM}{\mathcal M}         
\newcommand{\calN}{\mathcal N}         
         
         \newcommand{\frakP}{\mathfrak P}

\newcommand{\la}{\lambda}
\newcommand{\veps}{\varepsilon}

\newcommand{\bmat}{\begin{pmatrix}}
\newcommand{\emat}{\end{pmatrix}}
\newcommand{\mat}[4]
{
   \begin{pmatrix}
      #1 & #2\\
      #3 & #4
   \end{pmatrix}
}

\newcommand{\smallvek}[2]{\left(\begin{smallmatrix}#1\\#2\end{smallmatrix}\right)}

\newcommand{\dom}{\operatorname{dom}}

\newcommand{\ol}{\overline}
\newcommand{\ds}{\dotplus}

\newcommand{\eig}{\operatorname{eig}}
\newcommand{\sig}{\operatorname{sig}}

\begin{document}
\vspace*{-.3cm}
\begin{center}
\begin{spacing}{1.7}
{\LARGE\bf Eigenvalues in gaps of selfadjoint operators in Pontryagin spaces}
\end{spacing}

\vspace{1cm}
{\Large Friedrich Philipp}

\vspace{.3cm}
Technische Universit\"at Ilmenau, Postfach 10 05 65, 98684 Ilmenau, Germany, fmphilipp@gmail.com
\end{center}

\vspace{.5cm}
\begin{abstract}
\noindent Given an open real interval $\Delta$ and two selfadjoint operators $A_1$, $A_2$ in a $\Pi_\kappa$-space with $n$-dimensional resolvent difference we show that the difference of the total multiplicities of the eigenvalues of $A_1$ and $A_2$ in $\Delta$ is at most $n+2\kappa$.
\end{abstract}

\vspace{.6cm}


\section{Introduction and main result}
For selfadjoint operators $A_1$ and $A_2$ in a Hilbert space with $n$-dimensional resolvent difference (that is
$$
\dim\left((A_1 - \la)^{-1} - (A_2 - \la)^{-1}\right) = n
$$
holds for some (and hence for all) $\la\in\rho(A_1)\cap\rho(A_2)$) it is well-known that for each open interval $\Delta\subset\R\setminus\sigma_{\rm ess}(A_1)$ we have
\begin{equation}\label{e:hs}
\big|\eig(A_1,\Delta) \,-\, \eig(A_2,\Delta)\,\big|\,\le\, n,
\end{equation}
where $\eig(A_j,\Delta)$ denotes the number of eigenvalues of $A_j$ in $\Delta$ (counting multiplicities), $j=1,2$. In this note it is our main objective to generalize this theorem to the situation where $A_1$ and $A_2$ are selfadjoint operators in a Pontryagin space (for a detailed study of Pontryagin spaces and operators therein we refer to the monographs \cite{ai,b,ikl}). Since in the proof for the Hilbert space case (see, e.g., \cite[\paragraf 9.3, Theorem 3]{bs}) it is essential that the underlying inner product is positive definite, it cannot be expected that the estimate \eqref{e:hs} holds in the Pontryagin space situation. And indeed, the following simple example shows that \eqref{e:hs} is not even true in a two-dimensional $\Pi_1$-space.

\begin{ex}
In the space $\C^2$ we define the matrices
$$
J := \mat{1}{0}{0}{-1},\quad A_1 := \mat{1}{i}{i}{-1}\quad\text{and }\,A_2 := \mat{1/2}{0}{0}{1}
$$
and the inner product $[x,y] := (Jx,y)$, $x,y\in\C^2$, where $\hproduct$ denotes the standard scalar product in $\C^2$. Both matrices $A_1$ and $A_2$ are obviously selfadjoint in $(\C^2,\product)$. Moreover, $A_1\smallvek{2}{i} = A_2\smallvek{2}{i} = \smallvek{1}{i}$. But $\sigma(A_1) = \{0\}$ while $\sigma(A_2) = \{1/2,1\}$ and hence $\eig(A_1,(1/4,2)) = 0$ while $\eig(A_2,(1/4,2)) = 2$.
\end{ex}

If $(\frakP,\product)$ is a Pontryagin space and $\calM\subset\frakP$ is a closed subspace, then we write $\sig(\calM) := \kappa_+(\calM) - \kappa_-(\calM)$, where $\kappa_+(\calM)$ ($\kappa_-(\calM)$) denotes the number of positive (negative) squares of the inner product $\product$ on $\calM$. With this definition our main result reads as follows.

\begin{thm}\label{t:main}
Let $A_1$ and $A_2$ be selfadjoint operators in a Pontryagin space $(\frakP,\product)$ with $n$-dimensional resolvent difference. Then for every open {\rm (}bounded or unbounded{\rm )} interval $\Delta\subset\R\setminus\sigma_{\rm ess}(A_1)$ we have
\begin{equation}\label{e:sig_diff}
|\sig(\calL_\Delta(A_2)) - \sig(\calL_\Delta(A_1))|\,\le\,n,
\end{equation}
where $\calL_\Delta(A_j)$ denotes the closed linear span of the root subspaces of $A_j$ corresponding to the eigenvalues of $A_j$ in $\Delta$, $j=1,2$. In particular, if $\kappa$ denotes the number of negative squares of the inner product $\product$ on $\frakP$, then
\begin{equation}\label{e:eig_diff}
|\eig(A_1,\Delta) - \eig(A_2,\Delta)|\,\le\,n + 2\kappa.
\end{equation}
\end{thm}

We prove the theorem in section \ref{s:proof}. Clearly, if the Pontryagin space is in fact a Hilbert space (i.e.\ $\kappa = 0$), then both \eqref{e:sig_diff} and \eqref{e:eig_diff} coincide with \eqref{e:hs}. Hence, Theorem \ref{t:main} is a generalization of the known Hilbert space result.

An isolated eigenvalue of a selfadjoint operator in a Krein space $(\calK,\product)$ is said to be of {\it positive type} if the corresponding eigenspace is a Hilbert space with respect to the inner product $\product$. The following corollary can be seen as a local version of the Hilbert space case.

\begin{cor}\label{c:Pont_+}
Let $\frakP$, $A_1$, $A_2$ and $\Delta$ be as in Theorem {\rm\ref{t:main}}. If the number of negative squares of $\product$ on the spaces $\calL_\Delta(A_1)$ and $\calL_\Delta(A_2)$ coincide, then we have
$$
|\eig(A_1,\Delta) - \eig(A_2,\Delta)| \,\le\, n.
$$
This holds in particular if the eigenvalues of $A_1$ and $A_2$ in $\Delta$ are of positive type.
\end{cor}

If $(\frakP,\product)$ is a Pontryagin space with $\kappa$ {\it positive} squares, then the application of Theorem \ref{t:main} to the Pontryagin space $(\frakP,-\product)$ also yields \eqref{e:sig_diff} and \eqref{e:eig_diff}. In the finite-dimensional case this leads to the following corollary.

\begin{cor}
Let $\frakP$, $A_1$, $A_2$ and $\Delta$ be as in Theorem {\rm\ref{t:main}}. If $\dim\frakP < \infty$, then
$$
\left|\eig(A_1,\Delta) - \eig(A_2,\Delta)\right| \,\le\, n + 2\min\{\kappa_+(\frakP),\kappa_-(\frakP)\}.
$$
\end{cor}

We conclude this section with an example which shows that equality in \eqref{e:eig_diff} is possible in the case $\kappa = 1$.

\begin{ex}
In the space $\C^3$ define the matrices
$$
J := \bmat -1 & 0 & 0\\0 & 1 & 0\\0 & 0 & 1\emat,\,
A_1 := \bmat 0 & 100i & 0\\100i & 0 & 0\\0 & 0 & 0\emat,\,
A_2 := \bmat 0 & 100i & 0\\100i & 400 & 20\\0 & 20 & 1\emat
$$
and the inner product $\product := (J\cdot,\cdot)$. We have
$$
A_1(1,0,0)^T = A_2(1,0,0)^T\quad\text{ and }\quad A_1(0,1,-20)^T = A_2(0,1,-20)^T.
$$
Hence, $n = 1$. Moreover, $\sigma(A_1) = \{100i,-100i,0\}$ and $A_2$ has three distinct eigenvalues in $(0,\infty)$.
\end{ex}

\section{Proof of the main result}\label{s:proof}
%
Recall that an open interval $(a,b)$ belongs to the resolvent set of a selfadjoint operator $T$ in a Hilbert space $(\calH,\hproduct)$ if and only if
$$
((T - a)x,(T - b)x)\,\ge\,0
$$
holds for all $x\in\dom T$. The same inequality with the opposite relation, i.e.
$$
((T - a)x,(T - b)x)\,\le\,0,
$$
holds for all $x\in\dom T$ if and only if $\sigma(T)\subset[a,b]$. These relations will be used below.

\begin{prop}\label{p:negpos}
Let $(\frakP,\product)$ be a Pontryagin space with $\kappa$ negative squares, let $A$ be a selfadjoint operator in $(\frakP,\product)$ and let $a,b\in\R$, $a < b$. Then the following holds:
\begin{itemize}
\item[{\rm (}a{\rm )}] If $[a,b]\subset\rho(A)$, then $\frakP$ admits a decomposition $\frakP = \calM_-\ds\calM_+$, where $\calM_-\subset\dom A$ and $\dim\calM_- = \kappa$ such that
$$
[(A - a)x,(A - b)x] < 0
\quad\text{for }\,x\in\calM_-\setminus\{0\}
$$
and
$$
[(A - a)x,(A - b)x] > 0
\quad\text{for }\,x\in(\calM_+\cap\dom A)\setminus\{0\}
$$
\item[{\rm (}b{\rm )}] If $\sigma(A)\subset (a,b)$, then $\frakP$ admits a decomposition $\frakP = \calM_-\ds\calM_+$, where $\dim\calM_- = \kappa$ such that
$$
[(A - a)x,(A - b)x] > 0
\quad\text{for }\,x\in\calM_-\setminus\{0\}
$$
and
$$
[(A - a)x,(A - b)x] < 0
\quad\text{for }\,x\in\calM_+\setminus\{0\}.
$$
\end{itemize}
\end{prop}
\begin{proof}
First of all we show that it is no restriction to assume that the operator $A$ is bounded. In case (b) this immediately follows from the condition $\sigma(A)\subset (a,b)$. In case (a) we choose a ball $B_r(0)$, $r > 0$, with the zero point in the center such that the non-real spectrum of $A$ is contained in $B_r(0)$ and $E(\R\setminus [-r,r])\frakP$ is a Hilbert space with respect to the inner product $\product$ (where $E$ denotes the spectral function of the operator $A$). The restriction of $A$ to this Hilbert space is selfadjoint. Hence $[(A - a)x,(A-b)x] > 0$ holds for all $x\in E(\R\setminus [-r,r])\frakP$, $x\neq 0$. Thus, if (a) holds for the bounded operator $A|E([-r,r])\frakP$, then it obviously also holds for $A$.

By a Theorem of L.S.\ Pontryagin (see also \cite[Theorem 12.1$^\prime$]{ikl}) there exists a $\kappa$-dimensional non-positive subspace $\calL\subset\frakP$ which is $A$-invariant. Choose a (negative) subspace $\calL_-\subset\calL$ such that $\calL = \calL_-\,[\ds]\,\calL^\iso$, where $\calL^\iso$ denotes the isotropic part of $\calL$. Evidently, $\calL^\iso$ is $A$-invariant. By \cite[Theorem IX.2.5]{b} (see also \cite[Theorem 3.4]{ikl}) there exist a subspace $\frakP_0\subset\dom A$ with $\dim\frakP_0 = \dim\calL^\iso$ and a (uniformly) positive subspace $\calM$ such that
$$
\frakP = \calL_-\,[\ds]\,(\calL^\iso\,\ds\,\frakP_0)\,[\ds]\calM.
$$
Since $\calL^\iso$, $\calL$ and $\calL^\gperp = \calL^\iso\ds\calM$ are $A$-invariant, with respect to the decomposition
$$
\frakP = \calL^\iso\,\ds\,\calL_-\,\ds\,\calM\,\ds\,\frakP_0
$$
the operator $A$ has the following operator matrix representation:
$$
A = \bmat
A_{11} & A_{12} & A_{13} & A_{14}\\
   0   & A_{22} &    0   & A_{24}\\
   0   &    0   & A_{33} & A_{34}\\
   0   &    0   &    0   & A_{44}
\emat.
$$
In both cases (a) and (b) we have $a,b\in\rho(A)$. Therefore, the inner product
$$
\langle x,y\rangle := [(A - a)x,(A - b)y],\quad x,y\in\frakP
$$
defines a Krein space inner product on $\frakP$.

We only consider the case (a) here. The proof of (b) follows analogous lines. For $\mathfrak m\in\calM$ we have
$$
\langle\mathfrak m,\mathfrak m\rangle = [A_{13}\mathfrak m + (A_{33} - a)\mathfrak m,A_{13}\mathfrak m + (A_{33} - b)\mathfrak m] = [(A_{33} - a)\mathfrak m,(A_{33} - b)\mathfrak m].
$$
From $(a-\veps,b+\veps)\subset\rho(A)\subset\rho(A_{33})$ for some $\veps > 0$ and the selfadjointness of $A_{33}$ in the Hilbert space $(\calM,\product)$ we conclude that
$$
[(A_{33} - (a - \veps))\mathfrak m,(A_{33} - (b + \veps))\mathfrak m]\,\ge\,0,
$$
and hence
$$
\langle\mathfrak m,\mathfrak m\rangle\,\ge\,\veps(b-a+\veps)[\mathfrak m,\mathfrak m],
$$
which shows that $\calM$ is uniformly $\langle\cdot,\cdot\rangle$-positive. Similarly, it is shown that $\calL_-$ is $\langle\cdot,\cdot\rangle$-negative. Moreover, $\calL^\iso$, $\calL_-$ and $\calM$ are mutually $\langle\cdot,\cdot\rangle$-orthogonal and $\calL^\iso$ is $\langle\cdot,\cdot\rangle$-neutral. Hence, $(\frakP,\langle\cdot,\cdot\rangle)$ is a Pontryagin space with $\kappa$ negative squares which proves the assertion.
\end{proof}

We are now ready to prove Theorem \ref{t:main}.

\begin{proof}[Proof of Theorem \ref{t:main}]
It is no restriction to assume that the number of eigenvalues of $A_1$ in $\Delta$ (counting multiplicities) is finite. By $E_j$ we denote the spectral function of the operator $A_j$, $j=1,2$. Let $\Delta' = (a,b)$ be a subinterval of $\Delta$ which contains all the eigenvalues of $A_1$ in $\Delta$ such that $[a,b]\subset\Delta$ and $a,b\in\rho(A_1)\cap\rho(A_2)$. According to Proposition \ref{p:negpos} for $j=1,2$ we have decompositions
$$
(I - E_j(\Delta'))\frakP = \calM_{+,\rm out}^j\,\ds\,\calM_{-,\rm out}^j
\quad\text{and}\quad
E_j(\Delta')\frakP = \calM_{+,\rm in}^j\,\ds\,\calM_{-,\rm in}^j,
$$
where $\calM_{-,\rm out}^j\subset\dom A_j$,
$$
\dim\calM_{-,\rm out}^j = \kappa_-((I - E_j(\Delta'))\frakP)
\quad\text{and}\quad
\dim\calM_{-,\rm in}^j = \kappa_-(E_j(\Delta')\frakP)
$$
such that
$$
[(A_j - a)x,(A_j - b)x]\,<\,0\quad\text{for }\,x\in(\calM_{-,\rm out}^j[\ds]\calM_{+,\rm in}^j)\setminus\{0\},
$$
and
$$
[(A_j - a)x,(A_j - b)x]\,>\,0\quad\text{for }\,x\in\left((\calM_{+,\rm out}^j\cap\dom A_j)[\ds]\calM_{-,\rm in}^j\right)\setminus\{0\}.
$$
Evidently,
$$
\frakP = \left(\calM_{+,\rm out}^1\,\ds\,\calM_{-,\rm out}^1\right)\;[\ds]\,\left(\calM_{+,\rm in}^1\,\ds\,\calM_{-,\rm in}^1\right).
$$
Let $Q_1$ be the projection onto $\calM_{-,\rm out}^1\,[\ds]\,\calM_{+,\rm in}^1$ with respect to this decomposition of $\frakP$. Moreover, set
$$
\calK := \left(\calM_{-,\rm out}^2\,[\ds]\,\calM_{+,\rm in}^2\right)\cap\calD,
$$
where
$$
\calD := \{x\in\dom A_1\cap\dom A_2 : A_1x = A_2x\}.
$$
Note that $\calM_{-,\rm out}^2\,[\ds]\,\calM_{+,\rm in}^2\subset\dom A_2$. Assume that there exists $x\in\calK$ with $Q_1x = 0$ and $x\neq 0$. From $x\in\calK$ we deduce
$$
[(A_1 - a)x,(A_1 - b)x] = [(A_2 - a)x,(A_2 - b)x]\,<\,0.
$$
But $Q_1x = 0$ implies $x\in\calM_{+,\rm out}^1\,[\ds]\,\calM_{-,\rm in}^1$ and hence
$$
[(A_1 - a)x,(A_1 - b)x]\,>\,0.
$$
A contradiction. Therefore, the restriction of the linear mapping $Q_1$ to $\calK$ is one-to-one which yields $\dim\calK\,\le\,\dim Q_1\frakP$, i.e.
$$
\dim\calK\,\le\,\dim\calM_{-,\rm out}^1 + \dim\calM_{+,\rm in}^1 = \kappa_-((I - E_1(\Delta'))\frakP) + \kappa_+(E_1(\Delta')\frakP).
$$
On the other hand, as $\dim(\dom A_2/\calD) = n$ it follows that
$$
\dim\calK\,\ge\,\dim\calM_{-,\rm out}^2 + \dim\calM_{+,\rm in}^2 - n = \kappa_-((I - E_2(\Delta'))\frakP) + \kappa_+(E_2(\Delta')\frakP) - n,
$$
and we obtain
\begin{align*}
\kappa_+(E_2(\Delta')\frakP) - \kappa_+(E_1(\Delta')\frakP)
&\le n + \kappa_-((I - E_1(\Delta'))\frakP) - \kappa_-((I - E_2(\Delta'))\frakP)\\
&= n + (\kappa - \kappa_-(E_1(\Delta')\frakP)) - (\kappa - \kappa_-(E_2(\Delta')\frakP))\\
&= n + \kappa_-(E_2(\Delta')\frakP) - \kappa_-(E_1(\Delta')\frakP).
\end{align*}
This implies $\sig(\calL_{\Delta'}(A_2)) - \sig(\calL_{\Delta'}(A_1))\,\le\,n$ and hence also
$$
\eig(A_2,\Delta')\,\le\, n + 2\kappa + \eig(A_1,\Delta).
$$
Now, it is clear that $\eig(A_2,\Delta)$ is finite, and the relations \eqref{e:sig_diff} and \eqref{e:eig_diff} follow.
\end{proof}

\end{document}